\documentclass[a4paper, 10pt]{article}

\usepackage{graphics,graphicx}
\usepackage{amssymb,amsmath}
\usepackage{enumerate}

\usepackage{color}

\newtheorem{theorem}{Theorem}[section]

\newtheorem{lemma}{Lemma}[section]

 \def \sm {\setminus}

\newenvironment{proof}[1][]%
{\noindent {\setcounter{equation}{0}\it Proof.
}{#1}{}}{\hfill$\Box$\vspace{2ex}}

\def\longbox#1{\parbox{0.85\textwidth}{#1}}

\begin{document}

\title{Structural domination and coloring of some ($P_7, C_7$)-free graphs}

\author{S. A. Choudum\thanks{Department of Mathematics, CHRIST (Deemed to be University), Bengaluru 560029, India. Email: sac@retiree.iitm.ac.in}\and
T. Karthick\thanks{Corresponding author. Computer Science Unit, Indian Statistical
Institute, Chennai Centre, Chennai 600029, India. Email: karthick@isichennai.res.in}
\and%
Manoj M. Belavadi\thanks{Department of Mathematics, CHRIST (Deemed to be University), Bengaluru 560029, India. Email: manoj.belavadi@maths.christuniversity.in}}

\date{\today}

\maketitle

\begin{abstract}
   We show that every connected induced subgraph of a graph $G$ is dominated by an induced connected split graph if and only if $G$ is $\cal{C}$-free, where $\cal{C}$ is a set of six graphs which includes $P_7$ and $C_7$, and each   containing an induced $P_5$. A similar characterisation is shown for the class of graphs which are dominated by induced complete split graphs. Motivated by these results, we study  structural descriptions of some classes of $\cal{C}$-free graphs. In particular, we give  structural descriptions for the class of ($P_7$,$C_7$,$C_4$,\,gem)-free graphs and for the class of ($P_7$,$C_7$,$C_4$,\,diamond)-free graphs. Using these results, we show that every ($P_7$,$C_7$,$C_4$,\,gem)-free graph $G$ satisfies $\chi(G) \leq 2\omega(G)-1$, and that every ($P_7$,$C_7$,$C_4$,\,diamond)-free graph $H$ satisfies $\chi(H) \leq \omega(H)+1$. These two upper bounds are tight for any subgraph of the Petersen graph containing a $C_5$.

\end{abstract}

\section{Introduction}

Throughout the paper, we consider only simple and finite graphs. For a positive integer $\ell\geq 1$, let $P_\ell$ and $K_\ell$ respectively denote the chordless path and the complete graph on $\ell$ vertices. For a positive integer $\ell\geq 3$, $C_\ell$ is the chordless cycle on $\ell$ vertices. A \emph{diamond} is the four-vertex complete graph minus an edge, and a \emph{gem} is the graph consisting of a $P_4$ (say $P$) plus a vertex which is adjacent to all the vertices of $P$. Given a graph $H$, we say that a graph $G$ is \emph{$H$-free} if it does not contain an induced subgraph isomorphic to $H$.  Given a set $\{L_1,L_2,\ldots\}$ of graphs, we say that $G$ is \emph{$(L_1,L_2,\ldots)$-free} if $G$ is $L_i$-free, for each $i$.

 A \emph{clique} (\emph{stable set}) in a graph $G$ is a set of mutually adjacent (non-adjacent) vertices in $G$.
A \emph{split graph} is a graph whose vertex set can be partitioned into a stable set and a clique. A \emph{complete split graph} is a split graph where every vertex of the stable set is adjacent to every vertex of the clique.

 Domination and vertex colorings are two of the major topics extensively studied in graph theory. It is reflected in large number of books, monographs, and periodic surveys.
 In a graph $G$, a subset $D$ of $V(G)$ is a \emph{dominating set}  if every vertex in $V(G)\setminus D$ is adjacent to some vertex in $D$. The \emph{dominating induced subgraph} of a graph $G$ is the subgraph induced by a dominating set in $G$. For
any integer $k$, a \emph{$k$-coloring} of a graph $G$ is a mapping
$\phi:V(G)\rightarrow\{1,\ldots,k\}$ such that any two adjacent vertices
$u,v$ in $G$ satisfy $\phi(u)\neq \phi(v)$.  The \emph{chromatic number} $\chi(G)$ of
a graph $G$ is the smallest integer $k$ such that $G$ admits a $k$-coloring.  It is interesting to note that in many proof techniques employed to obtain optimal coloring of graphs, the structure of  dominating sets is exploited; see \cite{CKS2007} for examples.   We refer to a classical work of Bacs\'o \cite{Bacso2009} for a theoretical foundation of structural domination.

Wolk \cite{Wolk1962} showed that every connected induced subgraph of a graph $G$ is dominated by the graph $K_1$ if and only if $G$ is ($P_4$,\,$C_4$)-free.
 Bacs\'o and Tuza \cite{BacsoTuza1990}, and Cozzens and Kellehar \cite{CozzensKelleher1990} independently showed that every connected induced subgraph of a graph $G$ is dominated by a complete graph if and only if $G$ is ($P_5$,\,$C_5$)-free. Pim van't Hof and Paulusma \cite{HofPaulusma10} showed that every connected induced subgraph of a graph $G$ is dominated by a complete bipartite graph (not necessarily induced) if and only if $G$ is ($P_6$,\,$C_6$)-free. Bacs\'o, Michalak, and Tuza \cite{BacsoMichalakTuza2005} characterized the classes of graphs which are dominated by bipartite graphs, cycles, stars, and complete $k$-partite graphs. For other related results, we refer to \cite{BacsoMichalakTuza2005, BacsoTuzaVoigt2007,HofPaulusma10} and the references therein.
In this paper, by employing similar techniques as in \cite{BacsoMichalakTuza2005}, we show that every connected induced subgraph of a graph $G$ is dominated by an induced connected split graph if and only if $G$ is $\cal{C}$-free, where $\cal{C}$ is a set of six graphs which includes $P_7$ and $C_7$, and each containing an induced $P_5$. A similar characterisation is shown for the class of graphs which are dominated by  induced complete split graphs. Motivated by these results, we study the structural descriptions of some classes of $\cal{C}$-free graphs. In particular, we give  structural descriptions for the class of ($P_7$,$C_7$,$C_4$,\,gem)-free graphs and for the class of ($P_7$,$C_7$,$C_4$,\,diamond)-free graphs. Using these results, we show that every ($P_7$,$C_7$,$C_4$,\,gem)-free graph $G$ satisfies $\chi(G) \leq 2\omega(G)-1$, and that every ($P_7$,$C_7$,$C_4$,\,diamond)-free graph $H$ satisfies $\chi(H) \leq \omega(H)+1$. These two upper bounds are tight for any subgraph of the Petersen graph containing a $C_5$.

A family $\cal{G}$ of graphs is \emph{$\chi$-bounded} \cite{Gyarfas1987} with binding function $f$ if $\chi(H) \leq f(\omega(H))$ holds whenever $G\in \cal{G}$ and $H$ is an induced subgraph of $G$. Thus the class of  ($P_7$,$C_7$,$C_4$,\,gem)-free graphs, and the class of ($P_7$,$C_7$,$C_4$,\,diamond)-free graphs are $\chi$-bounded.
 In this respect, we note the following existing results which are relevant to this paper, and we refer to a recent extensive survey \cite{survey} for several families of graphs which admit a $\chi$-binding function.
  \begin{itemize}\setlength\itemsep{-0.5pt}
       \item Every ($P_6$,\,diamond)-free graph $G$ satisfies $\chi(G)\leq \omega(G)+3$ \cite{CHM2018}.
    \item Every ($P_6$,\,gem)-free graph $G$ satisfies $\chi(G)\leq 8\omega(G)$ \cite{CKS2007}.
     \item Every ($P_6$,\,$C_4$)-free graph $G$ satisfies $\chi(G)\leq \lceil\frac{5\omega(G)}{4}\rceil$ \cite{KM-SIDMA}.
  \item Every ($P_7,C_4,C_5$)-free graph $G$ satisfies $\chi(G)\leq \frac{3\omega(G)}{2}$ \cite{2018}.
  \item Since the class of $P_7$-free graphs admits a $\chi$-binding function ($f(x)=6^{x-1}$) \cite{Gyarfas1987}, the class of $\cal{C}$-free graphs is $\chi$-bounded. Also we note that the problem of obtaining a polynomial $\chi$-binding function for the class of $P_7$-free graphs is open, and is open even for the class of ($P_5,C_5$)-free graphs; the best known $\chi$-binding function for such class of graphs is $f(x)=2^{x-1}$ \cite{Chud-Siva}.

  \item Since the class of $P_5$-free graphs does not admit a linear $\chi$-binding function \cite{Fouquet1995}, it follows that the class of $\cal{C}$-free graphs too does not admit a linear $\chi$-binding function.
    \end{itemize}
Furthermore, the class of $P_7$-free graphs are of particular interest in algorithmic graph theory as well since the computational complexity of   {\textsc{Minimum  Dominating Set Problem}}, {\textsc{Maximum Independent Set Problem}}, and $k$-{\textsc{Colorability Problem}} (for $k\geq 3$) are unknown for the class of $P_7$-free graphs.

    \section{Notation and terminology}
 We follow West \cite{West} for standard terminology and notation.

     A \emph{hole} in a graph is an induced subgraph which is a cycle of length at least four. A hole is called \emph{even} if it has an even number of vertices. An \emph{even-hole-free graph} is a graph with no even holes.

     If $G_1$ and $G_2$ are two vertex disjoint graphs, then $G_1\cup G_2$ is the graph with vertex set $V(G_1)\cup V(G_2)$ and the edge set $E(G_1)\cup E(G_2)$. For any two disjoint subsets
$X$ and $Y$ of $V(G)$, we denote by $[X,Y]$, the set of edges with one end in $X$ and other end in $Y$.  We say that $X$ is
\emph{complete} to $Y$ or $[X,Y]$ is complete if every vertex in $X$
is adjacent to every vertex in $Y$; and $X$ is \emph{anticomplete} to
$Y$ if $[X,Y]=\emptyset$.  If $X$ is singleton, say $\{v\}$, we simply
write $v$ is complete (anticomplete) to $Y$. A \emph{clique-cutset} of a graph $G$ is a clique $K$ in $G$ such that $G \setminus K$ has more connected components than $G$.

In a graph $G$, the \emph{neighborhood} of a vertex $x$ is the set
$N_G(x)=\{y\in V(G)\setminus \{x\}\mid xy\in E(G)\}$; we drop the
subscript $G$ when there is no ambiguity.   The neighborhood $N(X)$ of a
subset $X \subseteq V(G)$ is the set $\{u \in V(G)\setminus X  \mid  u$ $\mbox{~is adjacent to a vertex of }X\}$.
A vertex of a graph is \emph{bisimplicial} if its neighborhood is the union of two cliques (not necessarily disjoint).
Given a set $S\subseteq V(G)$, $G[S]$
denote the subgraph of $G$ induced by $S$ in $G$, and for any $x\in V(G)\setminus S$, we denote the set $N(x)\cap S$ by $N_{S}(x)$. If $S \subseteq V(G)$, then a vertex $y\in V(G)\setminus S$ is a \emph{private neighbor} of some $x\in S$ if $y$ is adjacent to $x$ and is non-adjacent to every vertex in $S\setminus\{x\}$.  Given a connected graph $G$, a vertex $u$ of $G$ is a \emph{cut-vertex} of $G$ if $G\setminus u$ is disconnected.
Given a graph $G$, \emph{attaching a leaf to a vertex} $u$ of $G$  means  adding  a vertex $u'$ and an edge $uu'$ to $G$. 

A \emph{blowup} of a graph $H$ with vertices $v_1,v_2,\ldots,v_n$ is any graph $G$ such that $V(G)$ can
be partitioned into $n$ (not necessarily non-empty) cliques
$A_i$, $v_i\in V(H)$, such that $A_i$ is complete to $A_j$ if $v_iv_j\in E(H)$,
and $A_i$ is anticomplete to $A_j$ if $v_iv_j\notin E(H)$.
A blowup of a graph $H$ is a \emph{max-blowup} if $A_i$ is maximal for each $v_i\in V(H)$.
In a blowup $G$ of a graph $H$, if $A_i$ is not-empty for some $i$, then for convenience, we call one vertex of $A_i$ as $v_i$.


\section{Graphs dominated by split graphs}

 A $paw$ is a graph on four vertices $a,b,c,$ and $d$, and four edges $ab, bc, ca$ and $ad$. Let $H_1$ be the graph obtained from the paw by adding the vertex $e$ and an edge $de$, and $H_2$ be the graph obtained from the $H_1$ by adding an edge $ae$.

Let ${\cal{L}}_1:= \{P_5,C_5,C_4,H_1,H_2\}$ and let ${\cal{L}}_2:=\{P_4,C_4, \mbox{paw}\}$.

For the proof of our main result of this section, we rewrite a characterisation of split graphs where every forbidden graph is connected.
 A well known result of F\"oldes and Hammer \cite{split} states that a graph $G$ is a split graph if and only if it is ($2K_2, C_4, C_5$)-free. A recent result proved in \cite{DSM2016} states that a connected graph $G$ is $2K_2$-free if and only if it is $(P_5,H_1,H_2)$-free. Thus we have the following lemma.

\begin{lemma}\label{lem1}
A connected graph $G$ is a split graph if and only if it is ${\cal{L}}_1$-free.
\end{lemma}

\begin{lemma}\label{lem2}
A connected graph $G$ is a complete split graph if and only if it is ${\cal{L}}_2$-free.
\end{lemma}
\begin{proof} Using a result in \cite{cmplt}, it is easy to deduce that a graph $G$ is a complete split graph if and only if it is $(K_2\cup K_1, C_4)$-free. So it is sufficient to show that a connected graph $G$ is $K_2\cup K_1$-free if and only if $G$ is ($P_4$, paw)-free.
If $G$ is $K_2\cup K_1$-free, then obviously $G$ is ($P_4$, paw)-free as $P_4$ and paw both contain a $K_2\cup K_1$. Conversely, let $G$ be a ($P_4$, paw)-free graph. Suppose to the contrary that there exists a $K_2\cup K_1$ in $G$ with vertices $x,y,$ and $z$ such that $xy \in E(G)$. Then since $G$ is connected there exists a shortest path $P$ of length at least 2 between $x$ and $z$. Now since $G$ is $P_4$-free, $P$ is of length $2$, and thus there exists a vertex  $w(\neq y)$ such that $w$ is adjacent to both $x$ and $z$. But then $\{y,x,w,z\}$  induces either a $P_4$ or a paw, a contradiction. So $G$ is $K_2\cup K_1$-free and the lemma follows.
\end{proof}

\begin{figure}[h]
\centering
 \includegraphics[width=11cm]{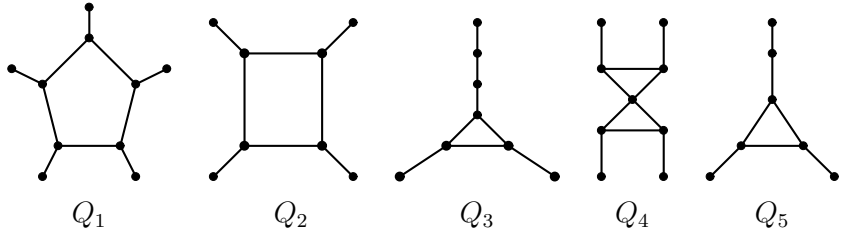}
\caption{Some special graphs}\label{fig:Q}
\end{figure}

Let $Q_1,Q_2,Q_3,Q_4$ and $Q_5$ be five graphs as shown in
Figure~\ref{fig:Q}.

Let ${\cal{C}}:= \{P_7,C_7,Q_1,Q_2,Q_3,Q_4\}$ and ${\cal{D}}:= \{P_6,C_6,Q_2,Q_5\}$.

 \begin{theorem}\label{splitgh}
Every connected induced subgraph of a graph $G$ is dominated by an induced connected split graph if and only if $G$ is $\cal{C}$-free.
 \end{theorem}
 \begin{proof}
If every connected induced subgraph of a graph $G$ is dominated  by an induced connected split graph, then $G$ is $\cal{C}$-free as none of the graphs in $\cal{C}$ is dominated by a connected split graph. To prove the converse, we may assume that $G$ is connected and it is sufficient to prove that $G$ is dominated by an induced connected split graph. In view of Lemma~\ref{lem1}, we prove that there exists a connected dominating  induced subgraph $H$ of $G$, where $H$ is ${\cal{L}}_1$-free.
\par
 Let $J$ be the set of all subgraphs $H$ of $G$ such that $H$ is a connected dominating induced subgraph of $G$. Then we claim the following.
  \begin{equation}\label{H-C4-free}
\mbox{There exists a graph in $J$ which is $C_4$-free}.
\end{equation}
Proof of $(\ref{H-C4-free})$: Suppose to the contrary that every graph in $J$ contains a $C_4$. Choose $H \in J$ such that $H$ contains minimum number of 4-cycles with as many leaves as possible attached to a 4-cycle, say $C$. We prove that each vertex $x\in V(C)$ has a leaf attached to it.
 If $x$ is a cut-vertex of $H$, then there exists a leaf attached to it. If $x$ is a non cut-vertex of $H$, then there exists a private neighbor $x_0$ of $x$ in $G$ (otherwise, $H\setminus x$ is a connected dominating subgraph with at least one less $C_4$ than $H$ which is a contradiction).  But then  $G[V(H)\cup \{x_0\}]$ is a connected dominating subgraph of $G$ with the same number of 4-cycles as $H$, but with more leaves attached to $C$ which is a contradiction. Hence every vertex in $C$ has a leaf attached to it.  This implies that $G$ contains $Q_2$, a contradiction to the fact that $G$ is $Q_2$-free. So $(\ref{H-C4-free})$ holds. $\lozenge$

\medskip
  Let $J_1\subseteq J$ be the set $\{H\in J \mid H$ is $C_4$-free$\}$. Then by $(\ref{H-C4-free})$,  $J_1$ is non-empty.  Then a similar proof as in $(\ref{H-C4-free})$  shows that the following holds.
 \begin{equation}\label{H-C5-free}
\mbox{There exists a graph in $J_1$ which is $C_5$-free}.
\end{equation}

 Let $J_2\subseteq J_1$ be the set $\{H\in J_1 \mid H$ is $C_5$-free$\}$. Then by  $(\ref{H-C5-free})$,  $J_2$ is non-empty.
 Then we claim the following.
  \begin{equation}\label{H-H1-free}
\mbox{There exists a graph in $J_2$ which is $H_1$-free}.
\end{equation}
Proof of $(\ref{H-H1-free})$: Suppose to the contrary that every graph in $J_2$ contains a $H_1$.
Let us choose $H\in J_2$ containing minimum number of copies of $H_1$ with as many leaves as possible attached to non cut vertices of a $H_1$; and we call this copy of $H_1$ as $M$. We prove that   each non cut-vertex $x\in V(M)$ has a leaf attached to it.  If $x$ is a cut-vertex of $H$, then there exists a leaf attached to it. If $x$ is a non cut-vertex of $H$, then there exists a private neighbour $x_0$ of $x$ in $G$ (otherwise, $H\setminus x\in J_2$  with less number of copies of $H_1$ than $H$ which is a contradiction). Moreover, $G[V(H)\cup \{x_0\}]\in J_2$ as adding $x_0$ to $H$ does not induce a $C_4$ or a $C_5$. But then $G[V(H)\cup \{x_0\}]$ contains the same number of copies of $H_1$ as $H$ with more leaves attached to $M$ which is a contradiction.   Hence every non cut-vertex in $M$ has a leaf attached to it.  This implies that $G$ contains $Q_3$, a contradiction to the fact that $G$ is $Q_3$-free. So $(\ref{H-H1-free})$ holds. $\lozenge$

\medskip
 Let $J_3\subseteq J_2$ be the set $\{H\in J_2 \mid H$ is $H_1$-free$\}$. Then by  $(\ref{H-H1-free})$,  $J_3$ is non-empty.
 Then we claim the following.
  \begin{equation}\label{H-H2-free}
\mbox{There exists a graph in $J_3$ which is $H_2$-free}.
\end{equation}
Proof of $(\ref{H-H2-free})$: Suppose to the contrary that every graph in $J_3$ contains a $H_2$.
Let us choose $H\in J_3$ containing minimum number of copies of $H_2$ with as many leaves as possible attached to non cut vertices of a $H_2$; and we call this copy of $H_2$ as $M$. Then  we shall prove that  each non cut-vertex $x\in V(M)$ has a leaf attached to it.  If $x$ is a cut-vertex of $H$, then there exists a leaf attached to it. If $x$ is a non cut-vertex of $H$, then there exists a private neighbour $x_0$ of $x$ in $G$; for otherwise, $H\setminus x\in J_3$ contains less number of copies of $H_2$ than $H$, a contradiction. Also we see that $G[V(H)\cup \{x_0\}]\in J_3$ as adding $x_0$ to $H$ does not induce a $C_4$ or a $C_5$ or a $H_1$. But then $G[V(H)\cup \{x_0\}]$ will have same number of copies of $H_2$ as $H$, but with more leaves attached to $M$, a contradiction.  Hence every non cut-vertex in $M$ has a leaf attached to it.  This implies that $G$ contains $Q_4$, a contradiction to the fact that $G$ is $Q_4$-free. So $(\ref{H-H2-free})$ holds. $\lozenge$

\smallskip
  Since $G$ is ($P_7,C_7$)-free, we see that every minimal dominating subgraph of $G$ is $P_5$-free (see also \cite{1993}). So there exists a minimal dominating subgraph $H\in J_3$ which is ${\cal{L}}_1$-free. This completes the proof.
 \end{proof}

\begin{theorem}
Every connected induced subgraph of a graph $G$ is dominated by an induced complete split graph if and only if $G$ is $\cal{D}$-free.
\end{theorem}
\begin{proof}
If every connected induced subgraph of $G$ is dominated by  an induced complete split graph, then $G$ is ${\cal{D}}$-free as none of the graphs in ${\cal{D}}$ is dominated by a complete split graph. To prove the converse,  we may assume that $G$ is connected and it is sufficient to prove that $G$ is dominated by an induced complete split graph. In view of Lemma~\ref{lem2}, we prove that there exists a connected dominating induced subgraph $H$ of $G$, where $H$ is ${\cal{L}}_2$-free.

   Let $J$ be the set of all subgraphs $H$ of $G$ such that $H$ is a connected dominating induced subgraph of $G$. Let $J_1\subseteq J$ be the set $\{H\in J \mid H$ is $C_4$-free$\}$. A proof similar to that of Theorem~\ref{splitgh} proves that $J_1$ is not empty.
 Then we claim  there exists a graph in $J_1$ which is paw-free. Suppose not. Choose $H\in J_1$ containing a minimum number of copies of paw with as many leaves as possible attached to non cut vertices of a paw; and we call this copy of paw as $W$. We prove that each non cut-vertex $x$ of $W$ has a leaf attached to it.  If $x$ is a cut-vertex of $H$, then there exists a leaf attached to it. If $x$ is a non cut-vertex of $H$, then there exists a private neighbour $x_0$ of $x$ in $G$; for otherwise, $H\setminus x$ belongs to $J_1$, containing less number of copies of paw than $H$, a contradiction. Moreover,  $G[V(H)\cup \{x_0\}]\in J_1$ as adding $x_0$ to $H$ does not induce a $C_4$. But then $G[V(H)\cup \{x_0\}]$ contains same number of copies of paw as $H$ with more leaves attached to $W$, a contradiction.  Hence every non cut-vertex in $W$ has a leaf attached to it.  This implies that $G$ contains $Q_5$, a contradiction to the fact that $G$ is $Q_5$-free.

\smallskip
 Since $G$ is ($P_6,C_6$)-free, we see that every minimal dominating subgraph of $G$ is $P_4$-free (see also \cite{1993}). So there exists a minimal dominating subgraph $H\in J_1$ which is ${\cal{L}}_2$-free. This completes the proof.
 \end{proof}

\section{The class of ($P_7$,\,$C_7$,\,$C_4$,\,gem)-free graphs}

Let ${\cal F} := \{P_7,C_7,C_4,\text{gem}\}$.  In this section, we give a structural description of  ${\cal F}$-free graphs, and show that the class of ${\cal F}$-free graphs is $\chi$-bounded.

Let $F$ be a $C_6:=$ $v_1$-$v_2$-$v_3$-$v_4$-$v_5$-$v_6$-$v_1$.

Let $F_1$ be the graph obtained from $F$ by adding vertices $y_1,y_4$, and edges $y_1v_1,y_1y_4$ and $y_4v_4$.

Let $F_2$ be the graph obtained from $F$ by adding vertices $x_1,z_1$, and edges $x_1v_1,x_1v_2,x_1z_1,z_1v_1$ and $z_1v_4$.

Throughout this section, we follow the convention that if the set of vertices $\{a_1,a_2,a_3,a_4,x\}$ induces a gem, then $a_1$-$a_2$-$a_3$-$a_4$ is an induced path in the neighborhood of $x$.

\begin{theorem}\label{Thm:C6Prop}
  Let $G$ be a connected $\cal{F}$-free graph that contains an induced $C_6$, say $v_1$-$v_2$-$v_3$-$v_4$-$v_5$-$v_6$-$v_1$. Let $A$ be the vertex set of a max-blowup of $C_6$ contained in $G$ such that $v_i\in A_i$.
 Let $R=\{x\in V(G)\setminus A \mid x$ has no
neighbor in $A\}$. For each $i$, $i$ mod $6$, let:
\begin{eqnarray*}
X_i &=& \{x\in V(G)\setminus A \mid x \mbox{ has a neighbor in } A_i \mbox{ and a neighbor in } A_{i+1},
\mbox{ and is } \\
& &  \mbox{ anticomplete to }A\setminus (A_{i}\cup A_{i+1})\},\\
Y_i &=& \{x\in V(G)\setminus A \mid x \mbox{ has a neighbor in } A_i,
\mbox{ and is anticomplete to } A\setminus A_{i}\}, \\
Z_i &=& \{x\in V(G)\setminus A \mid x \mbox{ has a neighbor in } A_i  \mbox{ and a neighbor in } A_{i+3},
\mbox{ and is } \\
& & \mbox{ anticomplete to } A\setminus (A_{i}\cup A_{i+3})\}.
\end{eqnarray*}
Let $X:=X_1\cup\cdots\cup X_6$, $Y:=Y_1\cup\cdots\cup Y_6$,  $Z:=Z_1\cup\cdots\cup Z_6$, and $B :=X\cup Y\cup Z$.  Then the following properties hold for all $i$:
\begin{enumerate}[(a)]
\item\label{c6-a}
$B = N(A)$ and hence $V(G) = A\cup B\cup R$.
\item\label{c6-b} $X_i$ is complete to $A_i\cup A_{i+1}$.
\item\label{c6-c}
$[X_i, (X\setminus (X_{i}\cup X_{i+3}))\cup (Y\setminus (Y_{i+3}\cup Y_{i+4}))\cup Z_{i+2}\cup Z_{i-1}\cup R]$, $[Y_i,  (Y\setminus(Y_{i}\cup Y_{i+3})\cup (Z\setminus Z_{i})\cup R]$,
 $[Z_i, Z\setminus (Z_i\cup Z_{i+3})]$ are  empty.
\item\label{c6-d} One of $X_i$ and $X_{i+1}\cup X_{i-1}$ is empty, one of $X_i$ and $Y_{i+2}\cup Y_{i-1}$ is empty, and one of
$Y_i$ and $Y_{i+2} \cup Y_{i-2}$ is empty.
\item\label{c6-e} If $[X_i, X_{i+3}\cup Y_{i+3}\cup Y_{i+4}]$ is not empty,  then (i) $Z_{i+2}\cup Z_{i+5}$ is empty, and (ii) $N(A_{i-1})=A_i\cup A_{i-2}$ or $N(A_{i+2})=A_{i+1}\cup A_{i+3}$.
\item\label{c6-f1} If $G$ is $F_1$-free, then $[Y_i, Y_{i+3}]$ is empty.

\item\label{c6-f} If $G$ is $F_2$-free, then $[X_i, Z_i\cup Z_{i+1}]$ is empty.

\end{enumerate}
\end{theorem}
\begin{proof}
(\ref{c6-a}) Consider any $x\in V(G)\setminus (A\cup R)$.  For each $i$, let
$a_i$ be a neighbor of $x$ in $A_i$ (if any such vertex exists) and
$b_i$ be a non-neighbor of $x$ in $A_i$ (if any exists).  Let
$L:=\{i\mid a_i$ exists$\}$.   Then $L\neq\emptyset$ since $x\notin R$. If $\{i-1,i+1\}\subset L$ for some $i$, then since
$a_{i-1}$-$b_i$-$a_{i+1}$-$x$-$a_{i-1}$ is a $C_4$, $b_i$ does not exist and hence $x$ is complete to $A_i$, and so $i\in L$.
Now up to symmetry $L$ is one of the following sets (1)--(4). In each case we make an observation.\\
(1) $L=\{i\}$ or $\{i,i+3\}$ for some $i$.  Then $x\in Y_i\cup Z_i$.\\
(2) $L=\{i,i+1\}$ for some $i$. Then $x\in X_i$.\\
(3) $L=\{i-1, i,i+1\}$ for some $i$. Then $x$ is complete to $A_i$ as proved above. Moreover, $x$ is complete to $A_{i-1}\cup A_{i+1}$, for otherwise $\{b_{i-1},b_i,a_{i+1},x, a_i\}$ or $\{b_{i+2},b_{i+1},a_i,x, a_{i+1}\}$ induces a gem.
So $x$ can be added to $A_i$, contradicting the
maximality of $A$.\\
(4) $L=\{i-1, i,i+1,i+2\}$ for some $i$ or $|L|=6$. Then $\{a_{i-1}, a_i, a_{i+1}, a_{i+2},x\}$ induces
a gem for some $i$.

So we conclude that $x\in B$, and hence (a) holds. $\lozenge$

 (\ref{c6-b}) Any $x\in X_i$ is complete to $A_{i}\cup A_{i+1}$, for otherwise $\{b_{i-1},b_i,a_{i+1},x, a_i\}$ or $\{b_{i+2},b_{i+1},a_i,x, a_{i+1}\}$ induces a gem.  So (\ref{c6-b}) holds. $\lozenge$

 Note that for each $i$, $Z_i =Z_{i+3}$ by the symmetry of $C_6$ and $i$ is mod 6. Also by (\ref{c6-b}), any $x\in X_i$ is complete to $\{v_i, v_{i+1}\}$.
From now on, we use the following notation for each $i$. For a vertex $p\in Y_i$, we let
$a_i$ be a neighbor of $p$ in $A_i$. Also, for a vertex $q\in  Z_i$, we let
$a_i'$ and $a_{i+3}'$ be neighbors of $q$, respectively, in $A_i$  and $A_{i+3}$. We prove the statements (\ref{c6-c})--(\ref{c6-f}) for $i=1$; for other cases of $i$ the proof is similar.

(\ref{c6-c}) If $x\in X_1$ and $y\in X_2$, then $\{v_1,x,y,v_3,v_2\}$ induces a gem.
If $x\in X_1$ and  $y\in X_3\cup Y_3 \cup Z_3$, then $x$-$v_2$-$v_3$-$y$-$x$ or $x$-$v_2$-$a_3$-$y$-$x$ or $x$-$v_2$-$a_3'$-$y$-$x$ is a $C_4$. If $x\in X_1$ and  $y\in Y_1$, then $y$-$x$-$v_2$-$v_3$-$v_4$-$v_5$-$v_6$ is a $P_7$. If $x\in Y_1$ and $y\in Y_2\cup Z_2$, then
  $x$-$a_1$-$a_2$-$y$-$x$ or $x$-$a_1$-$a_2'$-$y$-$x$ is a $C_4$. If $x\in Y_1$ and $y\in Y_3$, then $y$-$a_3$-$v_4$-$v_5$-$v_6$-$a_1$-$x$-$y$ is a $C_7$.
 If $x\in X_1\cup Y_1$ and $y\in R$, then $v_3$-$v_4$-$v_5$-$v_6$-$v_1$-$x$-$y$ or $v_3$-$v_4$-$v_5$-$v_6$-$a_1$-$x$-$y$ is a $P_7$. If $x\in Z_1$ and $y\in Z_2$, then $x$-$a_1'$-$a_2'$-$y$-$x$ is a $C_4$. If $xy$ belongs to one of the sets listed in (\ref{c6-c}), then it is symmetric to one of the cases showed above. This proves (\ref{c6-c}). $\lozenge$

(\ref{c6-d}) It is sufficient to prove $X_i\cup Y_i$ or $X_{i+1}\cup Y_{i+2}$ is empty. If there are vertices $x\in X_1\cup Y_1$ and $y\in X_2\cup Y_3$, then $x$ has a neighbor $a_1\in A_1$, $y$ has a neighbor $a_3\in A_3$.  By (\ref{c6-c}), $x$ is not adjacent to $y$. But then $y$-$a_3$-$v_4$-$v_5$-$v_6$-$a_1$-$x$ is a $P_7$.  
 So (\ref{c6-d}) holds. $\lozenge$

(\ref{c6-e}) $(i)$ Let $x \in X_1$ and $y \in X_4\cup Y_4\cup Y_5$ be adjacent. Suppose to the contrary that $Z_3 (=Z_6)\neq \emptyset$, and let $z_3\in Z_3$.  Then by (\ref{c6-c}), $z_3$ is anti-complete to $\{x, y\}$.  But now $y$-$x$-$v_1$-$a_6'$-$z_3$-$a_3'$-$v_4$ is a $P_7$ or $y$-$x$-$v_1$-$a_6'$-$z_3$-$a_3'$-$v_4$-$y$ is a $C_7$, a contradiction.

$(ii)$ If $[X_i, X_{i+3}\cup Y_{i+3}]\neq \emptyset$, then by ($\ref{c6-e}$:(i)), $Z_{i-1}=\emptyset$, and by ($\ref{c6-d}$), $X_{i-1}\cup X_{i-2}\cup Y_{i-1}=\emptyset$. So $N(A_{i-1})=A_i\cup A_{i-2}$.  If $[X_i, Y_{i+4}]\neq \emptyset$, then by ($\ref{c6-e}$:(i)), $Z_{i+2}=\emptyset$, and by ($\ref{c6-d}$), $X_{i+1}\cup X_{i+2}\cup Y_{i+2}=\emptyset$. So $N(A_{i+2})=A_{i+1}\cup A_{i+3}$.
Thus (\ref{c6-e}) holds. $\lozenge$

(\ref{c6-f1}) Suppose that there is an edge $xy$ with $x\in Y_1$ and $y\in Y_4$. Let $a_1$ be a neighbor of $x$ in $A_1$, and $a_4$ be a neighbor of $y$ in $A_4$. Now $\{a_1,v_2,v_3,a_4,v_5,v_6,x,$ $y\}$ induces $F_1$. So (\ref{c6-f1}) holds. $\lozenge$

(\ref{c6-f}) Suppose, up to symmetry, that there are adjacent vertices $x\in X_1$ and $z\in Z_1$. Since $x$ is complete to $A_1\cup A_2$ (by (\ref{c6-b})), we have $x$ is complete to $\{a_1',v_2\}$. But then $\{a_1',v_2,v_3,a_4',v_5,v_6,x,z\}$ induces $F_2$. So (\ref{c6-f}) holds. $\lozenge$

This completes the proof of the theorem.
\end{proof}

\begin{theorem}\label{thm:F1F2}
 Let $G$ be a connected $\cal{F}$-free graph. Then the following hold:
 \begin{enumerate}[(i)]
 \item If $G$ contains an $F_1$, then $G$ has a bisimplicial vertex.
 \item If $G$ contains an $F_2$, then either $G$ has a clique cutset or $G$ has a bisimplicial vertex.
 \end{enumerate}
 \end{theorem}
\begin{proof} We may assume that $G$ contains an $F_1$ or an $F_2$ with the same vertex-set and edge-set as defined earlier.
Since both $F_1$ and $F_2$ contain an induced $C_6$, let $A$ be the vertex set of a max-blowup of $C_6$ contained in $G$ such that $v_i\in A_i$. Now we partition the vertex-set of $G$ as in Theorem~\ref{Thm:C6Prop}, and we use the properties in Theorem~\ref{Thm:C6Prop} with the same notation.
\smallskip

\noindent{\it Proof of $(i)$}: Suppose that $G$ contains $F_1$. Note that $y_1\in Y_1$ and $y_4\in Y_4$. Then by Theorem~\ref{Thm:C6Prop}(\ref{c6-d}), $X_5\cup X_3\cup Y_2\cup Y_6$ is empty, and  one of $X_1$ and $X_6$ is empty. We may assume, up to symmetry, that $X_6 =\emptyset$. Now we claim that $Z_3(=Z_6)=\emptyset$. Suppose there is a vertex $z\in Z_3(=Z_6)$, then $z$ has a neighbor  $a_3\in A_3$ and $a_6\in A_6$, and since $z$ is anticomplete to $\{y_1,y_4\}$ (by Theorem~\ref{Thm:C6Prop}(\ref{c6-c})), we have $y_1$-$v_1$-$a_6$-$z$-$a_3$-$v_4$-$y_4$-$y_1$ a $C_7$, which is a contradiction. So $Z_3=\emptyset$.
Now we see that $N(v_6)$ is a union of two cliques $A_1$ and $A_5\cup (A_6\setminus \{v_6\})$, and hence $v_6$ is a bisimplicial vertex. This proves (i).

\medskip
\noindent{\it Proof of $(ii)$}: Suppose that $G$ contains $F_2$. Assuming that $G$ has no clique cutset, we show that $G$ has a bisimplicial vertex. Note that $x_1\in X_1$ and $z_1\in Z_1$.
Then by Theorem~\ref{Thm:C6Prop}(\ref{c6-b}), $x_1$ is complete to $A_1\cup A_2$, and by Theorem~\ref{Thm:C6Prop}(\ref{c6-d}), $X_2\cup X_6\cup Y_3\cup Y_6$ is empty. Moreover, by the definition of $Z_1$, $z_1$ has a neighbor in $A_4$, say $a_4$. Now here too we claim that $Z_3=\emptyset$. Suppose there is a vertex $z_3\in Z_3$, then $z_3$ has a neighbor  $a_3\in A_3$ and $a_6\in A_6$, and since $z_3$ is anticomplete to $\{x_1,z_1\}$ (by Theorem~\ref{Thm:C6Prop}(\ref{c6-c})), we have $z_1$-$x_1$-$v_2$-$a_3$-$z_3$-$a_6$-$v_5$ a $P_7$, which is a contradiction. So $Z_3$ is empty. Next we claim that:
\begin{equation}\label{nos}
\mbox{$[X_5, (B\setminus X_5)\cup R]$  is empty}.
\end{equation}
Proof of $(\ref{nos})$: Suppose that there is an edge $pq$ with $p\in X_5$ and $q\in (B\setminus X_5)\cup R$. Then by Theorem~\ref{Thm:C6Prop} (\ref{c6-c}), we see that $q\in Z_2\cup Y_2$, and hence $q$ has a neighbor $a_2\in A_2$. Since $p$ is complete to $A_6$ (by Theorem~\ref{Thm:C6Prop} (\ref{c6-b})), $pv_6\in E$. From Theorem~\ref{Thm:C6Prop} (\ref{c6-c}) by choosing the appropriate sets, we see that $p$ is anticomplete to $\{x_1,z_1\}$, and $qz_1\notin E$. Hence $z_1$-$v_1$-$a_2$-$q$ is a $P_4$. Also $x_1$ is complete to $A_2$ (by Theorem~\ref{Thm:C6Prop} (\ref{c6-b})), and since $\{z_1,v_1,a_2,q,x_1\}$ can not induce a gem, we have $qx_1\notin E$. But now we see that $v_6$-$p$-$q$-$a_2$-$x_1$-$z_1$-$a_4$ is a $P_7$, which is a contradiction. This proves (\ref{nos}). $\lozenge$

\medskip
Next we claim that:
\begin{equation}\label{emps}
\mbox{$X_5$  is empty}.
\end{equation}
Proof of $(\ref{emps})$: Suppose $x_5\in X_5$. Let $Q$ be a connected component in $G[X_5]$ containing $x_5$. Then by (\ref{nos}), we see that $N(Q)= A_5\cup A_6$ which is a clique. This implies that $A_5\cup A_6$ is a clique cutset separating $Q$ from $G[V(G)\setminus (Q\cup A_5\cup A_6)]$, which is a contradiction. This proves (\ref{emps}).  $\lozenge$

\medskip
Now by the above observations we note that $N(v_6)$ is a union of two cliques $A_1$ and $A_5\cup (A_6\setminus \{v_6\})$, and hence $v_6$ is a bisimplicial vertex. This proves (ii).

Thus the proof of Theorem~\ref{thm:F1F2} is complete.
\end{proof}

Let $F_3$ be the graph obtained from $F$ by adding vertices $z_1$ and $r$, and edges $z_{1}v_1,z_{1}v_4$ and $rz_{1}$.

\begin{theorem}\label{thm:F3}
 Let $G$ be a connected $\cal{F}$-free graph that contains an induced $F_3$. Then $G$ has a clique-cutset or
  $G$ has a bisimplicial vertex or
 $G$ is a blowup of the Petersen graph.

 \end{theorem}
 \begin{proof}
     First note that  if $G$ contains either  $F_1$ or $F_2$, then by Theorem~\ref{thm:F1F2}, $G$ has a clique cutset or $G$ has a bisimplicial vertex, and the theorem holds. So we may assume that $G$ is ($F_1$,\,$F_2$)-free.  Moreover, we may assume that $G$ has no clique cutset. We may assume that $G$ contains an $F_3$ with the same vertex-set and edge-set as defined earlier. Since $F_3$ contains an induced $C_6$, let $A$ be the vertex set of a max-blowup of $C_6$ contained in $G$ such that $v_i\in A_i$. Now we partition the vertex-set of $G$ as in Theorem~\ref{Thm:C6Prop}, and we use the properties in Theorem~\ref{Thm:C6Prop} with the same notation. Note that $z_1\in Z_1$ and $r\in R$.

      Now if $[X_i, X_{i+3}\cup Y_{i+3}\cup Y_{i+4}]$ is not empty, for some $i$,  then,  by (\ref{c6-e}) (of Theorem~\ref{Thm:C6Prop}), either $N(A_{i-1})=A_i\cup A_{i-2}$ or $N(A_{i+2})=A_{i+1}\cup A_{i+3}$. Hence either $v_{i-1}$ or $v_{i+2}$ is a bisimplicial vertex, and we conclude the theorem. So we may assume that:
 $[X_i, X_{i+3}\cup Y_{i+3}\cup Y_{i+4}]=\emptyset$, for all $i$.  Then by (\ref{c6-f}) and (\ref{c6-c}),  we see that $[X_i, (B\setminus X_i)\cup R]$ is empty, for all $i$.
Hence $N(X_i)$=$A_i\cup A_{i+1}$ is a clique, for all $i$. Since $G$ has no clique cutset, it follows that:

\begin{equation}\label{F3-X-emp}
\mbox{$X_i$  is empty, for every $i$.}
\end{equation}

Moreover, we claim that:
\begin{equation}\label{F3-Z-nemp}
\mbox{$Z_i$  is not empty, for all $i$}.
\end{equation}
Proof of (\ref{F3-Z-nemp}): Clearly, $Z_1\neq \emptyset$. If $Z_2=\emptyset=Y_2$, then $v_2$ is a bisimplicial vertex, and we conclude the theorem. If $Z_2=\emptyset$ and
$Y_2\neq \emptyset$, then by (\ref{F3-X-emp}) and (\ref{c6-c}) we see that $[Y_2, (B\setminus (Y_2\cup Y_5)) \cup R]$ is empty. Since $A_2$ is not a clique cutset, it follows that $[Y_2,Y_5]\neq \emptyset$. So there is an edge, say $y_2y_5$ with  $y_2\in Y_2$ and $y_5\in Y_5$. Then by the definition of $Y_i$'s, $y_2$ has a neighbor in $A_2$, say $a_2$, and $y_5$ has a neighbor in $A_5$, say $a_5$. But now $\{v_1,a_2,v_3,v_4,a_5,v_6, y_2,y_5\}$ induces an $F_1$. So we conclude that $Z_2$ is not empty. Likewise, $Z_3$ is not empty. This proves (\ref{F3-Z-nemp}), by the symmetry of $C_6$. $\lozenge$

\medskip
 We say a vertex $z\in Z_i$ is \emph{pure} if $z$ is complete to $A_i\cup A_{i+3}$, and is \emph{good} if $z$ has a neighbor in $R$.

\smallskip
 Now we claim that, for each $i$:
 \begin{equation}\label{F3-ZGP}
\mbox{Every good vertex in $Z_i$ is pure}.
\end{equation}
Proof of $(\ref{F3-ZGP})$: We prove for $i=1$. Let $z$ be a good vertex in $Z_1$. So there exists a neighbor of $z$ in $R$, say $r_1$. If $z$ has a non-neighbor in $A_1$, say $b_1$, then since $z$ has a neighbor in $A_4$, say $a_4$, we see that $r_1$-$z$-$a_4$-$v_3$-$v_2$-$b_1$-$v_6$ is a $P_7$. So $z$ is complete to $A_1$. Likewise, $z$ is complete to $A_4$. This proves (\ref{F3-ZGP}). $\lozenge$

     Next we claim that, for each $i$:
 \begin{equation}\label{F3-ZGC}
\mbox{Every good vertex $z\in Z_i$ is complete to $Z_{i}\setminus \{z\}$}.
\end{equation}
Proof of $(\ref{F3-ZGC})$: We prove for $i=1$. Let $z$ be a good vertex in $Z_1$.   Suppose that there exists a vertex $z'\in Z_1$ which is non-adjacent to $z$. Then by definition, $z'$ has a neighbor in $A_1$, say $a_1$,  and a neighbor in $A_4$, say $a_4$. Since $z$ is  pure (by (\ref{F3-ZGP})), we see that  $z$-$a_1$-$z'$-$a_4$-$z$ is a $C_4$.  This proves (\ref{F3-ZGC}). $\lozenge$

     Next we claim that:
      \begin{equation}\label{F3-R2ZC}
\longbox{If some vertex $x$ in $R$ has neighbors in at least two
 of the sets $Z_1, Z_2$, and $Z_3$, then $x$ is complete to $Z$.}
\end{equation}
Proof of $(\ref{F3-R2ZC})$:  We may assume, up to symmetry, that a vertex $x\in R$ is adjacent to $z_1\in Z_1$ and $z_2\in Z_2$. Then by (\ref{F3-ZGP}), $z_1$ and $z_2$ are pure. First suppose that there exists a vertex $z_1'\in Z_1$ which is non-adjacent to $x$. By definition, $z_1'$ has a neighbor in $A_1$, say $a_1$,  and a neighbor in $A_4$, say $a_4$. Since $z_1$ is  pure,  by (\ref{F3-ZGC}), $z_1$ is complete to $\{a_1,a_4,z_1'\}$. Then since $z_2$ is pure, we see that $\{z_1,a_1,v_6,v_5,z_2,x,z_1',a_4\}$ induces an $F_2$. So $x$ is complete to $Z_1$. Likewise, $x$ is complete to $Z_2$. Next, suppose that there exists a vertex $z_3\in Z_3$ which is non-adjacent to $x$. By definition, $z_3$ has a neighbor in $A_3$, say $a_3$,  and a neighbor in $A_6$, say $a_6$. Then since $z_1$ is pure, we see that  $a_3$-$z_3$-$a_6$-$v_1$-$z_1$-$x$-$z_2$ is a $P_7$. So $x$ is complete to $Z_3$. Hence we conclude that $x$ is complete to $Z$. This proves (\ref{F3-R2ZC}). $\lozenge$

We define the following subsets of $R$:
\begin{eqnarray*}
T_j &=& \{x\in R \mid x \mbox{ has a neighbor in } Z_j,
\mbox{ and is anticomplete to } Z\setminus Z_{j}\}, \\
&& \mbox{for } j\in \{1,2,3\},\\
R^* &=& \{x\in R \mid x \mbox{ is complete to } Z\},
\mbox{ and } \\
W &=&  \{x\in R \mid x \mbox{ is anticomplete to } Z\}.
\end{eqnarray*}

Let $T=T_1\cup T_2\cup T_3$. Then by (\ref{F3-R2ZC}), it follows that $R=T\cup R^*\cup W$ is a partition of $R$.  Further, we have the following:
 \begin{equation}\label{F3-parRprop}
\longbox{(i) $R^*$ is a clique, and is complete to $Z$, (ii) $[T_j, W]=\emptyset$, for all $j$ and $[T_j, T_k]$ is empty, for all $j\neq k$, and
(iii)  $T\cup W$ is empty; so $R = R^*$.}
\end{equation}
Proof of $(\ref{F3-parRprop})$: By (\ref{F3-Z-nemp}), $Z_i\neq \emptyset$, for $i\in \{1,2,3\}$. Let $z_i \in Z_i$.

 $(i)$: If there are non-adjacent vertices, say $r_1$ and $r_2$ in $R^*$, then by definition of $R^*$, $\{r_1,r_2\}$ is complete to $\{z_1,z_2\}$.
 But then $r_1$-$z_1$-$r_2$-$z_2$-$r_1$ is a $C_4$. So $R^*$ is a clique. Moreover, $R^*$ is complete to $Z$ by definition. So $(i)$ holds.

$(ii)$: Let $j=1$ and suppose that there is an edge $xy$ with $x\in T_1$ and $y\in W\cup T_2$. Then $x$ is adjacent to some vertex of $Z_1$, say $z$.
If $y\in W$, then $z_2$ is anticomplete to $\{x,y\}$, and then since $z$ is pure (by(\ref{F3-ZGP})), we see that  $y$-$x$-$z$-$v_1$-$v_6$-$a_5$-$z_2$ is a $P_7$, where $a_5$ is a neighbor of $z_2$ in $A_5$. If $y\in T_2$, then $y$ has a neighbor in $Z_2$, say $z'$. Then since $z$ and $z'$ are pure, (by(\ref{F3-ZGP})), we see that $\{z,v_1,v_2,z', v_5,v_4,x,y\}$ induces an $F_1$. The other cases are symmetric. So $(ii)$ holds.

$(iii)$: Since $G$ is connected, $N(W) \subseteq R^*$. Since $R^*$ is a clique (by (i)), and since $G$ has no clique cutset, we conclude that $W=\emptyset$. Also, for each $j\in \{1,2,3\}$,  we have $N(T_j)\subseteq R^*\cup Z_j$, and $R^*\cup Z_j$ is a clique by the definition of $R^*$ and (\ref{F3-ZGC}). Moreover $R^*\cup Z_j$ separates $T_j$ from $V(G)\setminus (R^*\cup Z_j\cup T_j)$. Since $G$ has no clique cutset, it follows that $T_j=\emptyset$, for each $j$.
Hence $T=\emptyset$.  So $(iii)$ holds.

This completes the proof of (\ref{F3-parRprop}). $\lozenge$

\medskip

By (\ref{F3-parRprop})(iii), $R=R^*$. Since $r\in R$, $R$ is non empty. So by (\ref{F3-parRprop})(i), every vertex in $Z_i$ is a good vertex, and hence $Z_i$ is a clique, and is complete to $A_i\cup A_{i+3}$, for each $i\in \{1,2,3\}$ (by (\ref{F3-ZGP})).

Finally, we claim that:

 \begin{equation}\label{F3-Yemp}
\mbox{$Y$ is empty.}
\end{equation}
Proof of $(\ref{F3-Yemp})$:  By Theorem~\ref{Thm:C6Prop} (\ref{c6-c}) and (\ref{c6-f1}), we see that $[Y_i,R\cup (B\setminus (Y_i\cup Z_i))]$ is empty, for every $i$. So $N(Y_i)\subseteq A_i\cup Z_i$. Since $A_i\cup Z_i$ is a clique, and since $G$ has no clique cutset, we conclude that $Y_i$ is empty, for each $i$. So (\ref{F3-Yemp}) holds. $\lozenge$

\medskip
By (\ref{F3-X-emp}) and (\ref{F3-Yemp}), $X\cup Y$ is empty. We conclude that $G =[(\cup_{i=1}^6 A_i)\cup(\cup_{j=1}^{3}Z_j)\cup R^*]$ is a blowup of the Petersen graph.\end{proof}

We use the following theorem.
\begin{theorem}[\cite{chud-bisimp}]\label{EHfree-bisimp}
Every even-hole-free graph has a bisimplicial vertex.
\end{theorem}

\begin{theorem}\label{thm:C6G}
 Let $G$ be a connected $\cal{F}$-free graph. Then $G$ has a clique-cutset or
  $G$ has a bisimplicial vertex or
 $G$ is a blowup of the Petersen graph.

 \end{theorem}
 \begin{proof}
First suppose that $G$ is $C_6$-free. Then since $G$ is ($P_7$,\,$C_4$)-free, $G$ is even-hole free. So by Theorem~\ref{EHfree-bisimp}, $G$ has a bisimplicial vertex.  So we may assume that $G$ contains an induced $C_6$, say $v_1$-$v_2$-$v_3$-$v_4$-$v_5$-$v_6$-$v_1$. Moreover, if $G$ contains either $F_1$ or  $F_2$ or $F_3$, then  the theorem follows by Theorems~\ref{thm:F1F2} and \ref{thm:F3}. So we may assume that $G$ is ($F_1$,\,$F_2$,\,$F_3$)-free. We may also assume that $G$ has no clique cutset. Now since $G$ contains an induced $C_6$, let $A$ be the vertex set of a max-blowup of $C_6$ contained in $G$ such that $v_i\in A_i$. Now we partition the vertex-set of $G$ as in Theorem~\ref{Thm:C6Prop}, and we use the properties in Theorem~\ref{Thm:C6Prop} with the same notation. 
Then by a similar proof of Theorem~\ref{thm:F3}, $X$ is empty. Further, we have the following:

 \begin{equation}\label{C6G-ZiEmp}
\mbox{$Z_i$ is empty, for some $i\in \{1,2,3\}$.}
\end{equation}
Proof of $(\ref{C6G-ZiEmp})$:  Suppose for each $i$, there exists $z_i\in Z_i$. By the definition of $Z_i$'s, for each $i$, $z_i$ has a neighbor in $A_i$, say $a_i$, and a neighbor in $A_{i+3}$, say $a_{i+3}$. But now  $\{z_1,a_1,a_2,z_2,a_5,a_4,a_6,z_3\}$ induces an $F_3$. So (\ref{C6G-ZiEmp}) holds. $\lozenge$

\medskip

We may assume, up to symmetry, that $Z_1$ is empty. Next we claim that:
\begin{equation}\label{C6G-Y1Emp}
\mbox{$Y_1$ is empty.}
\end{equation}
Proof of $(\ref{C6G-Y1Emp})$:   Since $X\cup Z_1$ is empty, we see that $[Y_1,R\cup (B\setminus Y_1)]$ is empty by Theorem~\ref{Thm:C6Prop} (\ref{c6-c}) and (\ref{c6-f1}). So $N(Y_1) \subseteq A_1$. Since $A_1$ is a clique, and since $G$ has no clique cutset, we conclude that $Y_1$ is empty.  So (\ref{C6G-Y1Emp}) holds. $\lozenge$

\medskip
Since $X\cup Y_1\cup Z_1$ is empty, we conclude that $N(A_1)=A_2\cup A_6$, and hence $v_1$ is a bisimplicial vertex in $G$. This completes the proof of the theorem.
 \end{proof}

To prove our next theorem we require the following result.
 \begin{theorem}[\cite{KM-SIDMA}]\label{thm:bu-peter}
If $G$ is any blowup of the Petersen graph, then $\chi(G)\le
\lceil \frac{5}{4} \omega(G)\rceil$.
\end{theorem}

 \begin{theorem}\label{2w-1chrbound}
Every  $\mathcal{F}$-free graph $G$ satisfies  $\chi(G)\leq 2\omega(G)-1$.
\end{theorem}
\begin{proof}
  Let $G$ be any $\mathcal{F}$-free graph. We prove the theorem by induction on $|V(G)|$.  We may assume that $G$ is connected, and we apply Theorem~\ref{thm:C6G}.

  If $G$ has a clique cutset $K$, let $A,B$ be a partition of
$V(G)\setminus K$ such that both $A,B$ are non-empty and
$[A,B]=\emptyset$. Clearly $\chi(G)=\max\{\chi(G[K\cup A]),
\chi(G[K\cup B])\}$, so the desired result follows from the induction
hypothesis on $G[K\cup A]$ and $G[K\cup B]$.

If $G$ has a bisimplicial vertex $u$, then $u$ has degree at most $2\omega(G)-2$ . By induction hypothesis, we have
$\chi(G\sm u)\le 2\omega(G\sm u)-1 \leq 2\omega(G)-1$.  So we can
take any $\chi(G\sm u)$-coloring of $G\sm u$ and extend it to a
($2\omega(G)-1$)-coloring of $G$, using for $u$ a
 color that does not appear in its neighborhood.

 If $G$ is a blowup of the Petersen graph, then the theorem follows by Theorem~\ref{thm:bu-peter}.

 This completes the proof of the theorem.
\end{proof}

\section{The class of ($P_7$,\,$C_7$,\,$C_4$,\,diamond)-free graphs}

A class of graphs $\cal{G}$ is said to satisfy the \emph{Vizing bound} if $\chi(G) \leq \omega(G)+1$, for each $G\in \cal{G}$. Several classes of graphs satisfying the Vizing bound are known in the literature, see for example \cite{survey}. Let $\mathcal{H} := \{P_7,C_7,C_4, \text{diamond}\}$. In this section, we give a structural description of  ${\cal H}$-free graphs, and show that the class of ${\cal H}$-free graphs satisfies the Vizing bound.

\begin{theorem}\label{cor:diamond}
 Let $G$ be a connected $\cal{H}$-free graph that contains an induced $C_6$. Then $G$ has a clique-cutset or
  $G$ has a vertex of degree at most $2$ or $G$ is the Petersen graph.
\end{theorem}
\begin{proof}
By applying the proof techniques of Theorems \ref{Thm:C6Prop}, \ref{thm:F1F2} and \ref{thm:F3}, this theorem follows.
\end{proof}

\begin{theorem}\label{thm:struc-diamond-C5noC6}
Let $G$ be a connected $\cal{H}$-free graph that contains an induced $C_5$. Suppose that $G$ is $C_6$-free. Then either  $G$ has a clique-cutset or $G$ has a  vertex of degree at most $\omega(G)$.
\end{theorem}
\begin{proof}
 We denote the $C_5$ contained in $G$ by $C$:=$a_1$-$a_2$-$a_3$-$a_4$-$a_5$-$a_1$. We may assume that $G$ has no clique cutset. Throughout the proof,  we assume  all subscripts are mod $5$.  For each $i$, $i\in \{1,2,\ldots,5\}$, let $X_i =\{x\in V \mid N_C(x) = \{a_i, a_{i+1}\}\}$, $Y_i=\{x\in V\mid N_C(x) = \{a_i\} \}$, $B=N(C)$, $R = V\setminus (B\cup C)$, $X:=\cup_{i=1}^{5}X_i$ and $Y:= \cup_{i=1}^{5}Y_i$. Then for each $i$,  the following hold:
\begin{enumerate}[(1)]\itemsep=1pt
\item\label{prop-diamond-1} $B = X\cup Y$; otherwise, $\{a_i,a_{i+2}\}\subseteq N_C(x)$ for some $x\in B$. Then $\{a_{i},a_{i+1},a_{i+2},x\}$ induces a $C_4$ or a diamond.

\item \label{prop-diamond-2} $G[X_i\cup\{a_i,a_{i+1}\}]$ is complete; if $u,v\in X_i$ are not adjacent, then $\{u,a_i,v,a_{i+1}\}$ induces a diamond.

\item \label{prop-diamond-3} $[X_i,(B\setminus (X_i\cup Y_{i+3})]$ and $[Y_i, Y\setminus Y_i]$ are empty.

Proof of (\ref{prop-diamond-3}): We prove for $i=1$. Let $xy$ be an edge in one of the sets. If $x\in X_1\cup Y_1$ and $y\in X_{2}\cup Y_{2}$, then $\{x,a_1,y,a_2\}$ induces a $C_4$ or a diamond. If $x\in X_1$ and $y\in X_{3}\cup Y_{3}$, then $x$-$a_2$-$a_3$-$y$-$x$ is a $C_4$. If $x\in Y_1$ and $y\in Y_3$, then $y$-$a_3$-$a_4$-$a_5$-$a_1$-$x$-$y$ is a $C_6$. The other cases of (\ref{prop-diamond-3}) follow by the symmetry of $C_5$. This proves (\ref{prop-diamond-3}). $\lozenge$
\end{enumerate}
Next to prove the theorem we consider two cases.
\par
Suppose that $[X_i, B\setminus X_i]\neq \emptyset$, for some $i$, say $i=1$.  Let $x\in X_1$ and $y\in B\setminus X_1$ be adjacent.  Then by (\ref{prop-diamond-3}), $y\in Y_4$. Moreover, $[X_5\cup Y_5,  (B\setminus (X_5\cup Y_5))\cup R]\subseteq [X_5\cup Y_5,X_2\cup Y_3\cup R]$ by (\ref{prop-diamond-3}). If $u\in X_5\cup Y_5$ is adjacent to $v\in X_2\cup Y_3\cup R$, then either $\{a_2,x,y,a_4,a_5,u,v\}$  or $\{x,a_2,a_3,a_4,a_5,u,v\}$ induces a $P_7$ or a $C_7$. So $[X_5\cup Y_5,X_2\cup Y_3\cup R]=\emptyset$, and hence $N(X_5\cup Y_5)= \{a_5,a_1\}$, which is a clique. Since $G$ has no clique cutset, we conclude that  $X_5\cup Y_5$ is empty.
So $N(a_5) =\{a_1,a_4\}\cup X_4$. Since  $\{a_4\}\cup X_4$ is a clique (by (\ref{prop-diamond-2})), we see that $\text{deg}(a_5)\leq \omega(G)$.
\par
Suppose that $[X_i, B\setminus X_i] = \emptyset$, for all $i$.  Now if $R =\emptyset$, then we prove that $G$ is a $C_5$. Indeed, if $x\in X_i\cup Y_i$, for some $i$, and if $Q$ is a component in $G[X_i\cup Y_i]$ containing $x$, then by (\ref{prop-diamond-3}), $N(Q) \subseteq \{a_i,a_{i+1}\}$, which is a clique. Since $G$ has no clique cutset, we conclude that $X_i\cup Y_i=\emptyset$, for all $i$, and hence $G$ is a $C_5$. So $\text{deg}(a_1)=2= \omega(G)$.
 Next assume that $R \neq \emptyset$. Then there exists a vertex $r\in R$ which is adjacent to some $x\in B$. We may assume that $x\in X_1\cup Y_1$. If $y\in X_2\cup Y_3$, then, by (\ref{prop-diamond-3}),  $\{y,a_3,a_4,a_5,a_1,x, r\}$ induces a $P_7$ or a $C_7$. So $X_2\cup Y_3 =\emptyset$. This implies that $N(a_3) =\{a_2,a_4\}\cup X_3$. Since  $\{a_4\}\cup X_3$ is a clique (by (\ref{prop-diamond-2})), we see that $\text{deg}(a_3)\leq \omega(G)$.  This completes the proof.
\end{proof}

\begin{theorem}\label{cor:struc-diamond}
Let $G$ be a connected $\cal{H}$-free graph. Then $G$ is the Petersen graph or $G$ has a clique-cutset or $G$ has a  vertex of degree at most $\omega(G)$.
\end{theorem}
\begin{proof} By Theorem~\ref{cor:diamond} and \ref{thm:struc-diamond-C5noC6}, we may assume that $G$ is ($C_5$,\,$C_6$)-free. Since $G$ is $P_7$-free, $G$ has no induced cycle of length at least $8$. Then since  $G$ is ($C_4$,\,$C_7$)-free, we conclude that $G$ is chordal, and so $G$ has a simplicial vertex of degree at most $\omega(G)$.
\end{proof}

\begin{theorem}\label{cor:col-diamond}
Every $\cal{H}$-free graph $G$ satisfies   $\chi(G)\le \omega(G)+1$.
\end{theorem}
\begin{proof}   Let $G$ be any $\mathcal{H}$-free graph. We prove the theorem by induction on $|V(G)|$.  We may assume that $G$ is connected, and we apply Corollary~\ref{cor:struc-diamond}.

If $G$ is the Petersen graph, then the theorem holds obviously, and if $G$ has a clique cutset, then the desired result follows as in Theorem~\ref{2w-1chrbound}.

If $G$ has a vertex $u$ of degree at most $\omega(G)$, then by induction hypothesis, we have
$\chi(G\sm u)\le \omega(G\sm u)+1 \leq \omega(G)+1$.  So we can
take any $\chi(G\sm u)$-coloring of $G\sm u$ and extend it to a
($\omega(G)+1$)-coloring of $G$, using for $u$ a
 color that does not appear in its neighborhood.

 This completes the proof.
\end{proof}


\section*{Acknowledgements}
The first  and the third authors thank CHRIST (Deemed to be University) for providing all the facilities to do this research work. The second author's research is partially supported by DST-SERB, Government of India under MATRICS scheme.

\end{document}